\documentclass{article}%
\usepackage{amssymb}
\usepackage{amsmath}
\usepackage{amsfonts}
\usepackage{graphicx}%
\providecommand{\U}[1]{\protect\rule{.1in}{.1in}}
\newtheorem{theorem}{Theorem}[section]

\newtheorem{corollary}[theorem]{Corollary}

\newtheorem{definition}[theorem]{Definition}

\newtheorem{remark}[theorem]{Remark}

\newenvironment{proof}[1][Proof]{\noindent\textbf{#1.} }{\ \rule{0.5em}{0.5em}}
\begin{document}

\title{A Survey of Some of the Recent Developments in Leavitt Path Algebras}
\author{Kulumani M. Rangaswamy\\Department of Mathematics, University of Colorado, \\Colorado Springs, CO. 80918}
\date{}
\maketitle

\section{Introduction}

Leavitt path algebras are algebraic analogues of graph C*-algebras and, ever
since they were introduced in 2004, have become an active area of research.
Many of the initial developments during the 2004 - 2014 period have been
nicely described in the recent book \cite{AAS} and in the excellent survey
article \cite{A}. Our goal in this article is to report on some of the recent
developments in the investigation of the algebraic aspects of Leavitt path
algebras not included in \cite{AAS}, \cite{A}. Because the Leavitt path
algebras grew as algebraic analogues of graph C*-algebras, their initial
investigation involved mostly the ideas and techniques used in the study of
graph C*-algebras such as the graph properties of Conditions (K) and (L), and
the ring properties of being simple, purely infinite simple, prime/primitive
etc. An important \ starting goal in this initial study was to work out the
algebraic analogue of the deep and powerful Kirchberg Phillips theorem to
classify purely infinite simple Leavitt path algebras $L:=L_{K}(E)$ up to
isomorphism or up to Morita equivalence by means the Grothendieck groups
$K_{0}(L)$ and the sign of the determinant $\det(I-A_{E})$ where $A_{E}$ is
the adjacent matrix of the graph $E$. After such initial progress, there has
been an explosion of articles dealing with not only the various different
aspects of Leavitt path algebras, but also many natural generalizations such
as Leavitt path algebras over commutative rings, of separated graphs, of high
rank graphs, Steinberg algebras and groupoids etc. In the background of many
of these investigation is the special feature that every Leavitt path algebra
$L$ is endowed with three mutually compatible structures: $L$ is a
$K$-algebra, $L$ is a $%
%TCIMACRO{\U{2124} }%
%BeginExpansion
\mathbb{Z}
%EndExpansion
$-graded ring and $L$ is a ring with involution $^{\ast}$. Our focus in this
survey is to describe a selection of recent graded and non-graded
ring-theoretic and module-theoretic investigations of Leavitt path algebras.
My apologies to authors whose work has not been included due to the
constrained focus, limitation of time and length of the paper.

In the first part of this survey, we describe graphical conditions \ on $E$
under which the corresponding Leavitt path algebra $L_{K}(E)$ belongs to
well-known classes of rings. The interesting fact is that often a single graph
property of $E$ seems to imply multiple ring properties of $L_{K}(E)$ and
these properties for general rings are usually independent of each other. The
poster child of such a phenomenon is the graph property for a finite graph $E$
that no cycle in $E$ has an exit. In this case $L_{K}(E)$ possesses at least
nine completely different ring properties ! (see Theorem \ref{NC-finiteCase}).
Because of such connections between $E$ and $L_{K}(E)$, Leavitt path algebras
can be effective tools in the construction of examples of rings with various
desired properties. If we do not impose any graphical conditions on $E$ and
just look at $L_{K}(E)$ as a $%
%TCIMACRO{\U{2124} }%
%BeginExpansion
\mathbb{Z}
%EndExpansion
$-graded ring, a really interesting result by Hazrat \cite{H} states that
$L_{K}(E)$ is a graded von Neumann regular ring. Because of this, the graded
one-sided and two-sided ideals of $L_{K}(E)$ possess many desirable properties.

The module theory over Leavitt path algebras is still at an infant stage.
\ The second part of this survey gives an account of some of the recent
advances in this theory. Naturally, the initial investigations focussed on the
simplest of the modules, namely, the simple modules over $L_{K}(E)$. We begin
with outlining a few methods of constructing graded and non-graded simple
left/right $L_{K}(E)$-modules. A special type of simple modules, called Chen
simple modules introduced by Chen \cite{C}, play an important role. This is
followed by characterizing Leavitt path algebras over which all the simple
modules possess some special properties, such as, when all the simple modules
are flat, or injective, or finitely presented or graded etc. For example, very
recently, A.A. Ambily, R. Hazrat and H. Li (\cite{AHL}) have proved that every
simple left/right $L_{K}(E)$-modules is flat if and only if $L_{K}(E)$ is von
Neumann regular, thus showing, in the case of Leavitt path algebras, an open
question in ring theory has an affirmative answer. Likewise, it was shown in
\cite{AMT-1} that $Ext_{L_{K}(E)}^{1}(S,S)\neq0$ for a Chen simple module $S$
induced by a cycle. It can then be shown that if all the simple left
$L_{K}(E)$-modules are injective, then $L_{K}(E)$ is von Neumann regular. The
converse easily holds if $E$ is a finite graph, since it that case $L_{K}(E)$
is semi-simple artinian. In contrast, if $R$ is an arbitrary non-commutative
ring, the injectivity of all simple left $R$- modules need not imply von
Neumann regularity of $R$ (see \cite{CO}). Our next result in this section
describes Leavitt path algebras of finite graphs having finite irreducible
representation type, that is, when there are only finitely many isomorphism
classes of simple modules. Interestingly, this class of algebras turns out to
be precisely the class of Leavitt path algebras of finite graphs having finite
Gelfand-Kirillov dimension.

The last section deals with one-sided ideals of a Leavitt path algebra $L$.
Four years ago it was shown in \cite{R-2} that finitely generated two-sided
ideals of $L$ are principal ideals. Recently, G. Abrams, F. Mantese and
A.Tonolo (\cite{AMT-2}) generalized this by showing that every finitely
generated one-sided ideal of $L$ is a principal ideal. Such rings are called
B\'{e}zout rings. Using a deep theorem of G. Bergman, Ara and Goodearl
(\cite{AG}) showed that one-sided ideals of $L$ are projective. From these two
results, it follows that the sum and the intersection of principal one-sided
ideals of $L$ are again principal. Thus the principal one-sided ideals of $L$
form a sublattice of the lattice of all one-sided ideals of $L$. A well-known
theorem, proved originally for graph C*-algebras and later for Leavitt path
algebras $L_{K}(E),$ states that every two-sided ideal of $L_{K}(E)$ is a
graded ideal if and only if $E$ satisfies Condition (K), equivalently
$L_{K}(E)$ is a weakly regular ring. What happens when every one-sided ideal
of $L_{K}(E)$ is graded? The last theorem of this section answers this
question, namely, every one-sided ideal of $L_{K}(E)$ is graded if and only if
every simple $L_{K}(E)$-module is graded if and only if $L_{K}(E)$ is a von
Neumann regular ring (see \cite{HR}).

In summary, this survey is intended to showcase a small sample of some of the
recent research on the algebraic aspects of Leavitt path algebras. Hopefully,
this provides the reader with some insights into this theory and generates
further interest in this exciting and growing field of algebra.

\ 

\section{Preliminaries}

For the general notation, terminology and results in Leavitt path algebras, we
refer to \cite{AAS} and \cite{A}. We give below an outline of some of the
needed basic concepts and results.

A (directed) graph $E=(E^{0},E^{1},r,s)$ consists of two sets $E^{0}$ and
$E^{1}$ together with maps $r,s:E^{1}\rightarrow E^{0}$. The elements of
$E^{0}$ are called \textit{vertices} and the elements of $E^{1}$
\textit{edges}. \textbf{\ }A vertex $v$ is called a \textit{sink} if it emits
no edges and a vertex $v$ is called a \textit{regular} \textit{vertex} if it
emits a non-empty finite set of edges. An \textit{infinite emitter} is a
vertex which emits infinitely many edges. For each $e\in E^{1}$, we call
$e^{\ast}$ a \textit{ghost} edge. We let $r(e^{\ast})$ denote $s(e)$, and we
let $s(e^{\ast})$ denote $r(e)$. A\textit{\ path} $\mu$ of length $n>0$ is a
finite sequence of edges $\mu=e_{1}e_{2}\cdot\cdot\cdot e_{n}$ with
$r(e_{i})=s(e_{i+1})$ for all $i=1,\cdot\cdot\cdot,n-1$. In this case
$\mu^{\ast}=e_{n}^{\ast}\cdot\cdot\cdot e_{2}^{\ast}e_{1}^{\ast}$ is the
corresponding ghost path. A vertex\ is considered a path of length $0$.

A path $\mu$ $=e_{1}\cdot\cdot\cdot e_{n}$ in $E$ is \textit{closed} if
$r(e_{n})=s(e_{1})$, in which case $\mu$ is said to be \textit{based at the
vertex }$s(e_{1})$. A closed path $\mu$ as above is called \textit{simple}
provided it does not pass through its base more than once, i.e., $s(e_{i})\neq
s(e_{1})$ for all $i=2,...,n$. The closed path $\mu$ is called a
\textit{cycle} if it does not pass through any of its vertices twice, that is,
if $s(e_{i})\neq s(e_{j})$ for every $i\neq j$. An \textit{exit }for a path
$\mu=e_{1}\cdot\cdot\cdot e_{n}$ is an edge $e$ such that $s(e)=s(e_{i})$ for
some $i$ and $e\neq e_{i}$.

For any vertex $v$, the \textit{tree} of $v$ is $T_{E}(v)=\{w\in E^{0}:v\geq
w\}$. We say there is a\textit{\ bifurcation} at a vertex $v$ or $v$ is
a\textit{\ bifurcation vertex}, if $v$ emits more than one edge. In a graph $E
$, a vertex $v$ is called a \textit{line point} if there is no bifurcation or
a cycle based at any vertex in $T_{E}(v)$. Thus, if $v$ is a line point, the
vertices in $T_{E}(v)$ arrange themselves on a straight line path $\mu$
starting at $v$ ($\mu$ could just be $v$) such as $\bullet_{v}\rightarrow
\bullet\cdot\cdot\cdot\bullet\rightarrow\bullet\cdot\cdot\cdot$ which could be
finite or infinite.

If $p$ is an infinite path in $E$, say, $p=e_{1}\cdot\cdot\cdot e_{n}%
e_{n+1}\cdot\cdot\cdot$ \ \ \ , we follow Chen \cite{C} to define, for each
$n\geq1$, $\tau^{\leq n}(p)=e_{1}\cdot\cdot\cdot e_{n}$ \ and $\tau
^{>n}(p)=e_{n+1}e_{n+2}\cdot\cdot\cdot$ \ . Two infinite paths $p,q$ are said
to be \textit{tail-equivalent} if there are positive integers $m,n$ such that
$\tau^{>m}(p)=\tau^{>n}(q)$. This defines an equivalence relation among the
infinite paths in $E$ and the equivalence class containing the path $p$ is
denoted by $[p]$. An infinite path $p$ is said to be a \textit{rational path
}if it is tail-equivalent to an infinite path $q=ccc\cdot\cdot\cdot$ , where
$c$ is a closed path.

Given an arbitrary graph $E$ and a field $K$, the \textit{Leavitt path algebra
}$L_{K}(E)$ is defined to be the $K$-algebra generated by a set $\{v:v\in
E^{0}\}$ of pair-wise orthogonal idempotents together with a set of variables
$\{e,e^{\ast}:e\in E^{1}\}$ which satisfy the following conditions:

(1) \ $s(e)e=e=er(e)$ for all $e\in E^{1}$.

(2) $r(e)e^{\ast}=e^{\ast}=e^{\ast}s(e)$\ for all $e\in E^{1}$.

(3) (The \textquotedblleft CK-1 relations\textquotedblright) For all $e,f\in
E^{1}$, $e^{\ast}e=r(e)$ and $e^{\ast}f=0$ if $e\neq f$.

(4) (The \textquotedblleft CK-2 relations\textquotedblright) For every regular
vertex $v\in E^{0}$,
\[
v=\sum_{e\in E^{1},s(e)=v}ee^{\ast}.
\]

An arbitrary element $a\in L:=L_{K}(E)$ can be written as $a=$ $%
%TCIMACRO{\dsum \limits_{i=1}^{n}}%
%BeginExpansion
{\displaystyle\sum\limits_{i=1}^{n}}
%EndExpansion
k_{i}\alpha_{i}\beta_{i}^{\ast}$ where $\alpha_{i},\beta_{i}$ are paths and
$k_{i}\in K$. Here $r(\alpha_{i})=s(\beta_{i}^{\ast})=r(\beta_{i})$.

Every Leavitt path algebra $L_{K}(E)$ is a $%
%TCIMACRO{\U{2124} }%
%BeginExpansion
\mathbb{Z}
%EndExpansion
$\textit{-graded algebra}, namely, $L_{K}(E)=%
%TCIMACRO{\dbigoplus \limits_{n\in\mathbb{Z}}}%
%BeginExpansion
{\displaystyle\bigoplus\limits_{n\in\mathbb{Z}}}
%EndExpansion
L_{n}$ induced by defining, for all $v\in E^{0}$ and $e\in E^{1}$, $\deg
(v)=0$, $\deg(e)=1$, $\deg(e^{\ast})=-1$. Here the $L_{n}$ are abelian
subgroups satisfying $L_{m}L_{n}\subseteq L_{m+n}$ for all $m,n\in%
%TCIMACRO{\U{2124} }%
%BeginExpansion
\mathbb{Z}
%EndExpansion
$. Further, for each $n\in%
%TCIMACRO{\U{2124} }%
%BeginExpansion
\mathbb{Z}
%EndExpansion
$, the homogeneous component $L_{n}$ is given by $L_{n}=\{%
%TCIMACRO{\tsum }%
%BeginExpansion
{\textstyle\sum}
%EndExpansion
k_{i}\alpha_{i}\beta_{i}^{\ast}\in L:\alpha,\beta\in Path(E),$ $|\alpha
_{i}|-|\beta_{i}|=n\}$. (For details, see Section 2.1 in \cite{AAS}). An ideal
$I$ of $L_{K}(E)$ is said to be a \textit{graded ideal} if $I=$ $%
%TCIMACRO{\dbigoplus \limits_{n\in\mathbb{Z}}}%
%BeginExpansion
{\displaystyle\bigoplus\limits_{n\in\mathbb{Z}}}
%EndExpansion
(I\cap L_{n})$.

Throughout this paper, $E$ will denote an arbitrary graph (with no restriction
on the number of vertices or on the number of edges emitted by each vertex)
and $K$ will denote an arbitrary field. For convenience in notation, we will
denote, most of the times, the Leavitt path algebra $L_{K}(E)$ by $L$.

We shall first recall the definition of the Gelfand-Kirillov dimension of
associative algebras over a field.

Let $A$ be a finitely generated $K$-algebra, generated by a finite dimensional
subspace $V=Ka_{1}\oplus\cdot\cdot\cdot\oplus Ka_{m}$. Let $V^{0}=K$ and, for
each $n\geq1$, let $V^{n}$ denote the $K$-subspace of $A$ spanned by all the
monomials of length $n$ in $a_{1},\cdot\cdot\cdot,a_{m}$. Set $V_{n}=%
%TCIMACRO{\tsum \limits_{i=0}^{n}}%
%BeginExpansion
{\textstyle\sum\limits_{i=0}^{n}}
%EndExpansion
V^{i}$. Then the \textbf{Gelfand-Kirillov dimension} of $A$ ( for short, the
\textbf{GK-dimension} of $A$) is defined by%
\[
\text{GK-dim}(A):=\underset{n\rightarrow\infty}{\lim\sup}\log_{n}(\dim
V_{n}).
\]
It is known that the GK-dim($A$) is independent of the choice of the
generating subspace $V$.

If $A$ is an infinitely generated $K$-algebra, then the GK-dimension of $A$ is
defined as
\[
\text{GK-dim}(A):=\underset{B}{Sup\text{ }}\text{GK-dim}(B)
\]
where $B$ runs over all the finitely generated $K$-subalgebras of $A$.

The GK-dimension of an algebra $A$ measures the growth of the algebra $A$ and
might be considered the non-commutative analogue of the classical Krull
dimension for commutative algebras. Some useful examples (see \cite{KL}) are:
The GK-dimension the matrix ring $M_{\Lambda}(K)$ is $0$ and the GK-dimension
of the matrix ring $M_{\Lambda^{\prime}}(K[x,x^{-1}])$ is $1$, where
$\Lambda,\Lambda^{\prime}$ are arbitrary possibly imnfinite index sets.

\textbf{Grading of a matrix ring over a }$%
%TCIMACRO{\U{2124} }%
%BeginExpansion
\mathbb{Z}
%EndExpansion
$\textbf{-graded ring:} We wish to recall the grading of matrices of finite
order and then indicate how to extend this to the case of infinite matrices in
which at most finitely many entries are non-zero (see \cite{H-1} and
\cite{NvO}). This will be used in Section 4. In the following, the length of a
path $p$ will be denoted by $|p|$.

Let $\Gamma$ be an additive abelian group, $A$ be a $\Gamma$-graded ring and
$(\delta_{1},\cdot\cdot\cdot,\delta_{n})$ an $n$-tuple where $\delta_{i}%
\in\Gamma$. Then $M_{n}(A)$ is a $\Gamma$-graded ring where, for each
$\lambda\in\Gamma$, its $\lambda$-homogeneous component consists of $n\times
n$ matrices%

\[
M_{n}(A)(\delta_{1},\cdot\cdot\cdot,\delta_{n})_{\lambda}=\left(
\begin{array}
[c]{cccccc}%
A_{\lambda+\delta_{1}-\delta_{1}} & A_{\lambda+\delta_{2}-\delta_{1}} & \cdot
& \cdot & \cdot & A_{\lambda+\delta_{n}-\delta_{1}}\\
A_{\lambda+\delta_{1}-\delta_{2}} & A_{\lambda+\delta_{2}-\delta_{2}} &  &  &
& A_{\lambda+\delta_{n}-\delta_{2}}\\
&  &  &  &  & \\
&  &  &  &  & \\
&  &  &  &  & \\
A_{\lambda+\delta_{1}-\delta_{n}} & A_{\lambda+\delta_{2}-\delta_{n}} &  &  &
& A_{\lambda+\delta_{n}-\delta_{n}}%
\end{array}
\right)  .\qquad\ \ (1)
\]

This shows that for each homogeneous element $x\in A$,
\[
\deg(e_{ij}(x))=\deg(x)+\delta_{i}-\delta_{j}\text{,}\qquad\qquad\qquad
\qquad(2)
\]
where $e_{ij}(x)$ is a matrix with $x$ in the $ij$-position and with every
other entry $0$.

Now let $A$ be a $\Gamma$-graded ring and let $I$ be an arbitrary infinite
index set. Denote by $M_{I}(A)$ the matrix with entries indexed by $I\times I
$ having all except finitely many entries non-zero and for each $(i,j)\in
I\times I$, the $ij$-position is denoted by $e_{ij}(a)$ where $a\in A$.
Considering a "vector" $\bar{\delta}:=(\delta_{i})_{i\in I}$ where $\delta
_{i}\in\Gamma$ and following the usual grading on the matrix ring \ (see
(1),(2)), define, for each homogeneous element $a$,%
\[
\deg(e_{ij}(a))=\deg(a)+\delta_{i}-\delta_{j}\text{.}\qquad\qquad\qquad
\qquad(3)
\]
This makes $M_{I}(A)$ a $\Gamma$-graded ring, which we denote by
$M_{I}(A)(\bar{\delta})$. Clearly, if $I$ is finite with $|I|=n$, then the
graded ring coincides (after a suitable permutation) with $M_{n}(A)(\delta
_{1},\cdot\cdot\cdot,\delta_{n})$.

Suppose $E$ is a finite acyclic graph consisting of exactly one sink $v$. Let
$\{p_{i}:1\leq i\leq n\}$ be the set of all paths ending at $v$. Then it was
shown in (Lemma 3.4, \cite{AAS-1})
\[
L_{K}(E)\cong M_{n}(K)\qquad\qquad\qquad(4)
\]
under the map $p_{i}p_{j}^{\ast}\longmapsto e_{ij}$. Now taking into account
the grading of $M_{n}(K)$, it was further shown in (Theorem 4.14, \cite{H-1})
that the same map induces a graded isomorphism
\[
L_{K}(E)\longrightarrow M_{n}(K)(|p_{1}|,\cdot\cdot\cdot,|p_{n}|)\qquad
\qquad\qquad(5)
\]%
\[
p_{i}p_{j}^{\ast}\longmapsto e_{ij\text{.}}%
\]
In the case of a comet graph $E$ (that is, a finite graph $E$ with a cycle $c
$ without exits and a vertex $v$ on $c$ such that every path in $E$ which does
not include all the edges in $c$ ends at $v$), it was shown in (Lemma 2.7.1,
\cite{AAS}) that the map%
\[
L_{K}(E)\longrightarrow M_{n}(K[x,x^{-1}])\qquad\qquad\qquad(6)
\]
given by
\[
\text{ }p_{i}c^{k}p_{j}^{\ast}\longmapsto e_{ij}(x^{k})
\]
where the $e_{ij}$ are matrix units, induces an isomorphism. Again taking into
account the grading, it was shown in (Theorem 4.20, \cite{H-1}) that the map%
\[
L_{K}(E)\longrightarrow M_{n}(K[x^{|c|},x^{-|c|}])(|p_{1}|,\cdot\cdot
\cdot\cdot,|p_{n}|)\qquad\qquad\qquad(7)
\]
given by
\[
\text{ }p_{i}c^{k}p_{j}^{\ast}\longmapsto e_{ij}(x^{k|c|})
\]
induces a graded isomorphism. Later in the paper \cite{AAS-2}, the
isomorphisms (4) and (6) were extended to infinite acyclic and infinite comet
graphs respectively \ (see Proposition 3.6 \cite{AAS-2}). The same
isomorphisms with the grading adjustments will induce graded isomorphisms for
Leavitt path algebras of such graphs. We now describe this extension below.

Let $E$ be a graph such that no cycle in $E$ has an exit and such that every
infinite path contains a line point or is tail-equivalent to a rational path
$ccc\cdot\cdot\cdot.$ \ where $c$ is a cycle (without exits). \ Define an
equivalence relation in the set of all line points in $E$ by setting $u\sim v
$ if $T_{E}(u)\cap T_{E}(v)\neq\emptyset$\ \ Let \ $X$ be the set of
representatives of distinct equivalence classes of line points in $E$, so that
for any two line points $u,v\in X$ with $u\neq v$, $T_{E}(u)\cap
T_{E}(v)=\emptyset$. For each vertex \ $v_{i}\in X$, let $\overset{-}{p^{v_{i}%
}}:=\{p_{s}^{v_{i}}:s\in\Lambda_{i}\}$ be the set of all paths that end at
$v_{i}$, where $\Lambda_{i}$ is an index set which could possibly be infinite.
Denote by $|\overset{-}{p^{v_{i}}|}=\{|p_{s}^{v_{i}}|:s\in\Lambda_{i}\}$.

Let $Y$ be the set of all distinct cycles in $E$. As before, for each cycle
$c_{j}\in Y$ based at a vertex $w_{j}$, let $\overset{-}{q^{w_{j}}}%
:=\{q_{r}^{w_{j}}:r\in\Upsilon_{j}\}$ be the set of all paths that end at
$w_{j}$ but do not include all the edges of $c_{j}$, where $\Upsilon_{j}$ is
an index set which could possibly be infinite. Let

$|$ $\overset{-}{q^{w_{j}}}|:=\{|q_{r}^{w_{j}}|:r\in\Upsilon_{j}\}$. Then the
isomorphisms (5) and (7) extend to a $%
%TCIMACRO{\U{2124} }%
%BeginExpansion
\mathbb{Z}
%EndExpansion
$-graded isomorphism%
\[
L_{K}(E)\cong_{gr}%
%TCIMACRO{\dbigoplus \limits_{v_{i}\in X}}%
%BeginExpansion
{\displaystyle\bigoplus\limits_{v_{i}\in X}}
%EndExpansion
M_{\Lambda_{i}}(K)(|\overset{-}{p^{v_{i}}|)}\oplus%
%TCIMACRO{\dbigoplus \limits_{w_{j}\in Y}}%
%BeginExpansion
{\displaystyle\bigoplus\limits_{w_{j}\in Y}}
%EndExpansion
M_{\Upsilon_{j}}(K[x^{|c_{j}|},x^{-|c_{j}|}])(|\overset{-}{q^{w_{j}}}%
|)\qquad(8)
\]
where the grading is as in (3).

\section{Leavitt path algebras satisfying a polynomial identity}

Observe that Leavitt path algebras in general are highly non-commutative. For
instance, if the graph $E$ contains an edge $\ e$ with $u=s(e)\neq r(e)=v $,
then $ue=e$, \ but $eu=0$. Indeed, it is an easy exercise to conclude that if
$\ E$ is a connected graph, then $L_{K}(E)$ is a commutative ring if and only
if the graph $E$ consists of just a single vertex $v$ or is a loop $e$, that
is a single edge $e$ with $s(e)=r(e)=v$. In this case, $L_{K}(E)$ is
isomorphic to $K$ or $K[x,x^{-1}]$.

Note that to say a ring $R$ is commutative is equivalent to saying that $R$
satisfies the polynomial identity $xy-yx=0$. \ An algebra $A$ over a field $K
$ is said to \textbf{satisfy a polynomial identity} (or simply, a
\textbf{PI-algebra}), if there is a polynomial $p(x_{1},\cdot\cdot\cdot
,x_{n})$ in finitely many non-commuting variable $x_{1},\cdot\cdot\cdot,x_{n}$
with coefficients in $K$ such that $p(a_{1},\cdot\cdot\cdot,a_{n})=0$ for all
choices of elements $a_{1},\cdot\cdot\cdot,a_{n}\in A$. For example, the
Amitsur-Levitzky theorem (see \cite{P}) states that the ring $M_{n}(R)$ of
$n\times n$ matrices over a commutative ring $R$ satisfies the so called
standard polynomial identity $P_{n}(x_{1},\cdot\cdot\cdot,x_{n})=%
%TCIMACRO{\dsum \limits_{\sigma\in S_{n}}}%
%BeginExpansion
{\displaystyle\sum\limits_{\sigma\in S_{n}}}
%EndExpansion
\epsilon_{\sigma}x_{\sigma(1)}\cdot\cdot\cdot x_{\sigma(n)}$ where $S_{n}$ is
the symmetric group of $n!$ permutations of the set $\{1,\cdot\cdot\cdot,n\}$
and $\epsilon_{\sigma}=1$ or $-1$ according as $\sigma$ is even or odd. A
natural question is to characterize the Leavitt path algebras which satisfy a
polynomial identity. This is completely answered in the next theorem.

\begin{theorem}
\label{BLR}(\cite{BLR}) Let $E$ be an arbitrary graph. Then the following
properties are equivalent for $L_{K}(E)$:

(a) $L_{K}(E)$ satisfies a polynomial identity;

(b) No cycle in $E$ has an exit, there is a fixed positive integer $d$ such
that the number of distinct paths that end at any vertex $v$ is $\leq d$ and
the only infinite paths in $E$ are paths that are eventually of the form
$ggg\cdot\cdot\cdot$, for some cycle $g$;

(c) There is a fixed positive integer $d$ such that $L_{K}(E)$ is a subdirect
product of matrix rings over $K$ or $K[x,x^{-1}]$ of order at most $d$.
\end{theorem}

If the graph $E$ is row-finite, then the Leavitt path algebra $L_{K}(E)$ in
Theorem \ref{BLR} actually decomposes as a ring direct sum of matrix rings
over $K$ or $K[x,x^{-1}]$ of order at most a fixed positive integer $d$. This
shows that satisfying a polynomial identity imposes a serious restriction on
the structure of Leavitt path algebras.

\section{Four Important Graphical Conditions}

In this section, we shall illustrate how specific graphical conditions on the
graph $E$ give rise to various algebraic properties of $L_{K}(E)$. We
illustrate this by choosing four different graph properties of $E$.
Interestingly, a single graph theoretical property of $E$ often implies
several different ring properties for $L_{K}(E)$. It is amazing that a single
property that no cycle in a finite graph $E$ has an exit implies that the
corresponding Leavitt path algebra $L_{K}(E)$ possesses several different ring
properties such as being directly finite, self-injective, having bounded index
of nilpotence, a Baer ring, satisfying a polynomial identity, having
GK-dimension $\leq1$, etc. (see Theorem \ref{NC-finiteCase} below).
Consequently, Leavitt path algebras turn out to be useful tools in the
construction of various examples of rings. We will also describe the
interesting history behind the terms Condition (K) and Condition (L) which
play an important role in the investigation of both the graph C*-algebras and
Leavitt path algebras (see \cite{AAS}, \cite{T},\cite{T-2}).

Recall, a ring $R$ is said to be von Neumann regular if to each element $a\in
R$ there is an element $b\in R$ such that $a=aba$. The ring $R$ is said to be
$\pi$-regular (strongly left or right $\pi$-regular) if to each element $a\in
R$, there is a $b\in R$ and an integer $n\geq1$ such that $a^{n}=a^{n}ba^{n}$
($a^{n}=a^{n+1}b$ or $a^{n}=ba^{n+1}$). In general, these ring properties are
not equivalent. But as the next theorem shows, they all coincide for Leavitt
path algebras.

A graph $E$ is said to be \textbf{acyclic} if $E$ contains no cycles. The next
theorem characterizes the von Neumann regular Leavitt path algebras.

\begin{theorem}
\cite{AR} For an arbitrary graph $E$, the following conditions are equivalent
for $L:=L_{K}(E)$:

(a) The graph $E$ is acyclic;

(b) $L$ is von Neumann regular;

(c) $L$ is $\pi$-regular;

(d) $L$ is strongly left/right $\pi$-regular.
\end{theorem}

Another important graph property is Condition (K). In some sense this property
is diametrically opposite of being acyclic.

\begin{definition}
A graph $E$ satisfies \textbf{Condition (K)} if whenever a vertex $v$ lies on
a simple closed path $\alpha$, $v$ also lies on another simple closed path
$\beta$ distinct from $\alpha$.
\end{definition}

The Condition (K) implies a number of ring properties;

\begin{definition}
(i) A ring $R$ is said to be left/right \textbf{weakly regular }if for every
left/right ideal $I$ of $R$, $I^{2}=I$;

(ii) A ring $R$ is said to be an \textbf{exchange ring} if given any
left/right $R$-module $M$ and two direct decompositions of $M$ as
$M=M^{\prime}\oplus A$ and $M=%
%TCIMACRO{\dbigoplus \limits_{i=1}^{n}}%
%BeginExpansion
{\displaystyle\bigoplus\limits_{i=1}^{n}}
%EndExpansion
A_{i}$, where $M^{\prime}\cong R$, there exist submodules $B_{i}\subseteqq
A_{i}$ such that $M=M^{\prime}\oplus%
%TCIMACRO{\dbigoplus \limits_{i=1}^{n}}%
%BeginExpansion
{\displaystyle\bigoplus\limits_{i=1}^{n}}
%EndExpansion
B_{i}$.
\end{definition}

\begin{theorem}
\label{Condn K}( \cite{APS}, \cite{ARS}, \cite{T}) Let $E$ be an arbitrary
graph. Then the following conditions are equivalent for $L:=L_{K}(E)$:

(i) \ The graph $E$ satisfies Condition (K);

(ii) $L$ is an exchange ring;

(iii) $L$ is left/right weakly regular;

(iv) Every two-sided ideal of $L$ is a graded ideal.
\end{theorem}

\begin{definition}
A graph $E$ is said to satisfy \textbf{Condition (L)}, if every cycle in $E$
has an exit.
\end{definition}

\begin{theorem}
\cite{R-1} Let $E$ be an arbitrary graph. Then the following are equivalent
for $L_{K}(E)$:

(i) $E$ satisfies Condition (L);

(ii) $L$ is a \textbf{Zorn ring}, that is, every (non-nil) right/left ideal
$I$ contains a non-zero idempotent.;

(iii) Every element $a\in L$ \ is the von Neumann inverse of another element
$b\in L$; that is, to each $a\in L$, there is an element $b\in L$ such that
$bab=b$.
\end{theorem}

\textbf{An interesting history of Conditions (K) and (L)}: One may wonder
about the choice of the letters K and L in the terms Condition (K) and
Condition (L). There is an interesting narrative about the origins of these
terms. I am grateful to Mark Tomforde for outlining this history to me which
he will also be including in his forth-coming book on Graph Algebras
(\cite{T-2}). Both these two graph conditions were originally introduced by
graph C*-algebraists. It all started when Cuntz and Krieger (whom some
consider the founders of graph C*-algebras), introduced in their original
paper (\cite{CK}) \ a condition on matrices with entries in $\{0,1\}$ and
called it Condition (I). Assuming that the "I" is the English letter I and not
the Roman numeral one, Pask and Raeburn introduced in 1996 a Condition (J) in
their paper (\cite{PR}), as J is the letter that follows I in the English
alphabet. (They apparently did not recognize that C\"{u}ntz and Krieger also
introduced a follow-up Condition (II), thus indicating, in their view, I and
II stand for Roman numerals). Conforming to this pattern, when Kumjian, Pask,
Raeburn and Renault introduced a new condition in their 1997 paper
(\cite{KPRR}), they chose the letter K to denote this new condition and called
it Condition (K). Continuing this pattern yet again, Kumjian, Pask and Raeburn
introduced Condition (L) in 1998 (\cite{KPR}). Actually, Astrid an Huef later
showed that Condition (L) coincides with Condition (I) for graphs of finite
matrices. Moreover, Condition (K) is considered analogous to Condition (II)
for Cuntz-Krieger algebras. In all the investigations that followed in graph
C*-algebras and also in Leavitt path algebras, Conditions (K) and (L) emerged
as important graph conditions. Poor Condition (J) remains neglected !

Recall, Condition (L) requires every cycle to have an exit. We next consider a
graph property that is diametrically opposite to Condition (L), namely, no
cycle in the graph has an exit. This implies several interesting ring/module properties.

First, consider a finite graph $E$ in which no cycle has an exit. In this
case, $L_{K}(E)$ is a ring with identity. \ We begin recalling a number of
ring properties.

A ring $R$ with identity $1$ is said to be \textbf{directly finite} if for any
two elements $x,y$, $xy=1$ implies $yx=1$. This is equivalent to $R$ being not
isomorphic to any proper direct summand of $R$ as a left or a right
$R$-module. A ring $R$ with identity is called a \textbf{Baer ring }if the
left/right annihilator of every subset $X$ of $R$ is generated by an
idempotent. A $\Gamma$-graded ring $R$ is said to be a \textbf{graded Baer
ring}, if the left/right annihilator of every subset $X$ of homogeneous
elements is generated by a homogeneous idempotent. A ring $R$ is said to have
\textbf{bounded index of nilpotence} if there is a positive integer $n$ which
is such that $a^{n}=0$ for every nilpotent element $a\in R$.

\begin{theorem}
\label{NC-finiteCase}( \cite{AAJZ}, \cite{BLR}, \cite{HR}, \cite{HRS-1},
\cite{HV}, \cite{V} ) For a finite graph $E$, the following conditions are
equivalent for $L:=L_{K}(E)$:

(i) \ \ \ \ \ No cycle in $E$ has an exit;

(ii) \ $\ \ \ L$ is directly finite;

(iii) $\ \ \ L$ is a Baer ring;

(iv) $\ \ \ L$ is a graded Baer ring;

(v) \ $\ \ L$ is a graded left/right self-injective ring;

(vi) $\ \ L$ satisfies a polynomial identity;

(vii) \ $L$ has bounded index of nilpotence;

(viii) $L$ is graded semi-simple;

(ix) \ $L$ has GK-dimension $\leq1;$

(x) \ $L$ is finite over its center.
\end{theorem}

Thus if $E$ is the following graph,
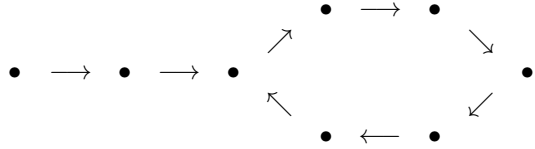
\[%
\begin{array}
[c]{ccccccccccc}
&  &  &  &  &  & \bullet & \longrightarrow & \bullet &  & \\
&  &  &  &  & \nearrow &  &  &  & \searrow & \\
\bullet_{{}} & \longrightarrow & \bullet & \longrightarrow & \bullet &  &  &
&  &  & \bullet\\
&  &  &  &  & \nwarrow &  &  &  & \swarrow & \\
&  &  &  &  &  & \bullet & \longleftarrow & \bullet &  &
\end{array}
\]
then $L_{K}(E)$ will possess all the stated nine ring properties.

For a finite graph $E$, if $L_{K}(E)$ satisfies any of the equivalent
conditions in the preceding theorem, $L_{K}(E)$ decomposes as a graded direct
sum of finitely many matrix rings of finite order over $K$ and/or
$K[x,x^{-1}]$ which are given the matrix gradings indicated in equations (5)
and (7) in the Preliminaries section.

For a ring $R$ without identity, but with local units, $R$ is said to be
directly finite if for every $x,y\in R$ and an idempotent $u\in R$ satisfying
$ux=x=xu,uy=y=yu$, we have $xy=u$ implies $yx=u$. Every commutative ring is
trivially directly finite.

If $R$ is a ring without identity, $R$ is called a \textbf{locally Baer ring}
(\textbf{locally graded Baer }ring) if for every idempotent (homogeneous
idempotent) $e$, the corner $eRe$ is a Baer (graded Baer) ring.

\begin{theorem}
(\cite{HR}, \cite{HV}) Let $E$ be an arbitrary graph. Then the following
conditions are equivalent for $L:=L_{K}(E)$:

(i) \ \ No cycle in $E$ has an exit, $E$ is row-finite and every infinite path
ends at a sink or a cycle.

(ii) $\ L$ is a locally Baer ring;

(iii) $L$ is a \ graded locally Baer ring;

(iv) $L$ is a graded left/right self-injective ring;

(v) $L$ is graded isomorphic to a ring direct sum of \ matrix rings%

\[
L_{K}(E)\cong_{gr}%
%TCIMACRO{\dbigoplus \limits_{v_{i}\in X}}%
%BeginExpansion
{\displaystyle\bigoplus\limits_{v_{i}\in X}}
%EndExpansion
M_{\Lambda_{i}}(K)(|\overset{-}{p^{v_{i}}|)}\oplus%
%TCIMACRO{\dbigoplus \limits_{w_{j}\in Y}}%
%BeginExpansion
{\displaystyle\bigoplus\limits_{w_{j}\in Y}}
%EndExpansion
M_{\Upsilon_{j}}(K[x^{t_{j}},x^{-t_{j}}])(|\overset{-}{q^{w_{j}}}|)
\]
where $\Lambda_{i}$,$\Upsilon_{j}$ are suitable index sets, the $t_{j}$ are
positive integers,\ $X$ is the set of representatives of distinct equivalence
classes of line points in $E$ and $Y$ is the set of all distinct cycles
(without exits) in $E$.
\end{theorem}

\section{ Simple modules over Leavitt path algebras}

In this section, we shall indicate the methods of constructing simple modules
over Leavitt path algebras by graphical methods.

As noted in \cite{AAS}, every element $a$ of a Leavitt path algebra $L_{K}(E)
$ over a graph $E$ can be written in the form $a=%
%TCIMACRO{\dsum \limits_{i=1}^{n}}%
%BeginExpansion
{\displaystyle\sum\limits_{i=1}^{n}}
%EndExpansion
\alpha_{i}\beta_{i}^{\ast}$ and that the map $%
%TCIMACRO{\dsum \limits_{i=1}^{n}}%
%BeginExpansion
{\displaystyle\sum\limits_{i=1}^{n}}
%EndExpansion
\alpha_{i}\beta_{i}^{\ast}\longrightarrow%
%TCIMACRO{\dsum \limits_{i=1}^{n}}%
%BeginExpansion
{\displaystyle\sum\limits_{i=1}^{n}}
%EndExpansion
\beta_{i}^{{}}\alpha_{i}^{\ast}$ induces an isomorphism $L_{K}%
(E)\longrightarrow(L_{K}(E))^{op}$. Consequently, $L_{K}(E)$ is left-right
symmetric. So in this and the next section, we shall only be stating results
on left ideals and left modules over $L_{K}(E)$. The corresponding results on
right ideals and right modules hold by symmetry.

\begin{definition}
A \ vertex $v$ is called a\textit{\ }\textbf{Laurent vertex }if $T_{E}(v)$
consists of the set of all vertices on a single path $\gamma=\mu c$ where
$\mu$ is a path without bifurcations starting at $v$ and $c$ is a \ cycle
without exits based on a vertex $u=r(\mu)$.
\end{definition}

An easy example of a Laurent vertex is the vertex $v$ in the following graph:%

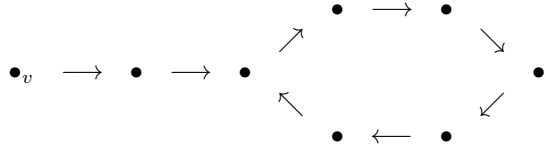
\[%
\begin{array}
[c]{ccccccccccc}
&  &  &  &  &  & \bullet & \longrightarrow & \bullet &  & \\
&  &  &  &  & \nearrow &  &  &  & \searrow & \\
\bullet_{v} & \longrightarrow & \bullet & \longrightarrow & \bullet &  &  &  &
&  & \bullet\\
&  &  &  &  & \nwarrow &  &  &  & \swarrow & \\
&  &  &  &  &  & \bullet & \longleftarrow & \bullet &  &
\end{array}
\]
The next theorem gives conditions under which a vertex in the graph $E$
generates a simple left ideal/graded simple left ideal of $L_{K}(E)$.

\begin{theorem}
(\cite{AAS}, \cite{HR}) Let $E$ be an arbitrary graph and let $v$ be a vertex. Then

(a) The left ideal $L_{K}(E)v$ is a simple/minimal left ideal of $L_{K}(E)$ if
and only if $v$ is a line point;

(b) The left ideal $L_{K}(E)v$ is a graded simple/minimal left ideal of
$L_{K}(E)$ if and only if $v$ is either a line point or a Laurent vertex.
\end{theorem}

Next we shall describe the general methodology used by Chen (\cite{C}) \ and
extended in (\cite{HR}, \cite{R-3}) to construct left simple and graded simple
modules over $L_{K}(E)$ by using special vertices or cycles in the graph $E$.

I) \textbf{Definition of the module }$A_{u}$: Let $u$ be a vertex in a graph
$E$ which is either a sink or an infinite emitter. Let $A_{u}$ be the
$K$-vector space having as a basis the set $B=\{p:p$ is a path in $E$ with
$r(p)=u\}$. We make $A$ a left $L_{K}(E)$-module as follows: Define, for each
vertex $v$ and each edge $e$ in $E$, linear transformations $P_{v},S_{e}$ and
$S_{e^{\ast}}$ on $A$ by defining their actions on the basis $B$ are as follows:

For all $p\in B$,

(I) $\ \ P_{v}(p)=\left\{
\begin{array}
[c]{c}%
p\text{, if }v=s(p)\\
0\text{, otherwise}%
\end{array}
\right.  $

(II) $\ S_{e}(p)=\left\{
\begin{array}
[c]{c}%
ep\text{, if }r(e)=s(p)\\
0\text{, \qquad otherwise}%
\end{array}
\right.  $

(III) $S_{e^{\ast}}(p)=\left\{
\begin{array}
[c]{c}%
p^{\prime}\text{, if }p=ep^{\prime}\\
0\text{, otherwise}%
\end{array}
\right.  $

(IV) $S_{e^{\ast}}(u)=0$.

The endomorphisms $\{P_{v},S_{e},S_{e^{\ast}}:v\in E^{0},e\in E^{1}\}$ satisfy
the defining relations (1) - (4) of the Leavitt path algebra $L_{K}(E)$. This
induces an algebra homomorphism $\phi$ from $L_{K}(E)$ to $End_{K}(A_{u})$
mapping $v$ to $P_{v}$, $e$ to $S_{e}$ and $e^{\ast}$ to $S_{e^{\ast}}$ . Then
$A_{u}$ can be made a left module over $L_{K}(E)$ via the homomorphism $\phi$.
We denote this $L_{K}(E)$-module operation on $A_{u}$ by $\cdot$.

\begin{theorem}
(\cite{C}, \cite{R-3}) If the vertex $u$ is either a sink or an infinite
emitter, then $A_{u}$ is a simple left $L_{K}(E)$-module.
\end{theorem}

If the vertex $u$ lies on a cycle without exits, then in the Definition of
$A_{u}$, define the basis $B=\{pq^{\ast}:p,q$ path in $E$ with $r(q^{\ast
})=s(q)=u\}$. We then get the following result.

\begin{theorem}
(\cite{HR}) If a vertex $u\in E$ lies on a cycle without exits, then $A_{u}$
is a graded simple left $L_{K}(E)$-module graded isomorphic to the graded
minimal left ideal $L_{K}(E)u$ and $A_{u}$ is not a simple left $L_{K}(E)$-module.
\end{theorem}

\begin{remark}
In \cite{HR}, the module $A_{u}$ is defined by using an algebraic branching
system and is denoted as $N_{vc}$. Here we have defined the module $A_{u}$
differently, but the proof of the above theorem is just the proof of Theorem
3.5(1) of \cite{HR}.
\end{remark}

With a slight modification of the definition of $A_{u}$, Chen \cite{C} shows
one more way of constructing simple modules by using the infinite paths that
are tail-equivalent to a fixed infinite path in $E$. Recall, two infinite
paths $p=e_{1}\cdot\cdot\cdot e_{r}\cdot\cdot\cdot$ and $q=f_{1}\cdot
\cdot\cdot f_{s}\cdot\cdot\cdot$ are said to be \textbf{tail-equivalent} if
there exist fixed positive integers $m,n$ such that $e_{n+k}=f_{m+k}$ for all
$k\geq1$. Let $[p]$ denote the tail-equivalence class of all infinite paths
equivalent to $p$. Let $A_{[p]}$ denote the $K$-vector space having $[p]$ as a
basis. As in the definition of $A_{u}$, for each vertex $v$ and each edge $e$
in $E$, define the linear transformations $P_{v},S_{e}$ and $S_{e^{\ast}}$ on
$A$ by defining their actions on the basis $[p]$ satisfying the conditions
(I),(II),(III), but not (IV) above. \ As before, they satisfy the defining
relations of a Leavitt path algebra and thus induce a homomorphism
$\varphi:L_{K}(E)\rightarrow A_{[p]}$. The vector space $A_{[p]}$ then becomes
a left $L_{K}(E)$-module via the map $\varphi$.

\begin{theorem}
\label{Chen}(\cite{C}) The module $A_{[p]}$ is a simple left $L_{K}(E)$-module
and for two infinite paths $p,q$, $A_{[p]}\cong A_{[q]}$ if and only if
$[p]=[q]$.
\end{theorem}

It can be shown (see \cite{R-3}) that the simple modules $A_{u},A_{v}$ and
$A_{[p]}$ corresponding respectively to a sink $u$, an infinite emitter $v$
\ and an infinite path $p$, are all non-isomorphic.

A special infinite path is the so called a \textbf{rational infinite path}
induced by a simple closed path (and in particular, a cycle) $c$. This is the
infinite path $ccc\cdot\cdot\cdot$ . We denote this path by $c^{\infty}$. We
shall be using the corresponding simple $L_{K}(E)$-module $A_{c^{\infty}}$ subsequently.

\section{Leavitt path algebras with simple modules having special properties}

We shall describe when all the simple modules over a Leavitt path algebra are
flat or injective or finitely presented or graded etc.

An open problem, raised by Ramamuthi (\cite{Ram}) some forty years ago, asks
whether a non-commutative ring $R$ with $1$ is von Neumann regular if all the
simple left $R$-modules are flat. Using a more general approach of Steinberg
algebras, Ambily, Hazrat and Li (\cite{AHL}) obtain the following theorem
which shows that Ramamurthi's question has an affirmative answer in the case
of Leavitt path algebras.

\begin{theorem}
(\cite{AHL}) Let $E$ be an arbitrary graph. Then every simple left $L_{K}%
(E)$-module is flat if and only if $L_{K}(E)$ is von Neumann regular.
\end{theorem}

Next we consider the case when $L_{K}(E)$ is a left V-ring, that is, when
every simple left $L_{K}(E)$-module is injective. Kaplansky showed that if $R$
is a commutative ring, then every simple $R$-module is injective if and only
if $R$ is von Neumann regular. A natural question is under what conditions
every simple left $L_{K}(E)$-module is injective. It was shown in \cite{HRS-2}
that, in this case, $L_{K}(E)$ becomes a weakly regular ring. However,
recently Abrams, Mantese and Tonolo (\cite{AMT-1}) showed that, if $c$ is a
cycle in a graph $E$, then the corresponding simple left $L_{K}(E)$-module
$A_{c^{\infty}}$ satisfies $Ext_{L_{K}(E)}^{1}(A_{c^{\infty}},$ $A_{c^{\infty
}})\neq0$. This implies that the module $A_{c^{\infty}}$ cannot be an
injective module. So, if every simple left $L_{K}(E)$-module is injective,
then necessarily $E$ contains no cycles. Then by (\cite{AR}) $L_{K}(E)$ must
be von Neumann regular. Thus we obtain the following new result and its corollary.

\begin{theorem}
Let $E$ \ be an arbitrary graph. If every simple left $L_{K}(E)$-module is
injective, then $L_{K}(E)$ is a von Neumann regular ring.
\end{theorem}

Conversely, if $L_{K}(E)$ is a von Neumann regular ring then is the graph $E$
contains no cycles (\cite{AR}) and, if $E$ is further a finite graph, then
$L_{K}(E)$ is a direct sum of finitely many matrix rings of finite order over
$K$ (Thorem 2.6.17, \cite{AAS}). In this case, $L_{K}(E)$ is a direct sum of
left/right simple modules and hence every simple left/right $L_{K}(E)$-module
is injective. This leads to the following corollary.

\begin{corollary}
Let $E$ be a finite graph. Then every simple left/right $L_{K}(E)$-module is
injective if and only if $L_{K}(E)$ is a von Neumann regular ring.
\end{corollary}

When $E$ is an arbitrary graph, it is an open question whether the von Neumann
regularity of $L_{K}(E)$ implies that every simple left/right $L_{K}%
(E)$-module is injective.

Next we consider Leavitt path algebras $L_{K}(E)$ whose simple modules are all
finitely presented. When $E$ is a finite graph, $L_{K}(E)$ possesses a number
of interesting properties as noted in the following theorem

\begin{theorem}
(\cite{A-1R}) For any finite graph $E$, the following properties of $\ $the
Leavitt path algebra $L:=L_{K}(E)$ are equivalent:

(i) \ \ Every simple left $L$-module \ is finitely presented;

(ii) \ No two cycles in $E$ have a common vertex;

(iii) There is a one-to-one correspondence between isomorphism classes of
simple $L$-modules and primitive ideals of $L$;

(iv) The Gelfand-Kirillov dimension of $L$ is finite.
\end{theorem}

The preceding theorem has been generalized in (\cite{R-4}) to the case when
$E$ is an arbitrary graph with several similar equivalent conditions.

\bigskip

It is easy to observe that every simple left module over a Leavitt path
algebra $L:=L_{K}(E)$ is of the form $Lv/N$ for some vertex $v$ and a maximal
left submodule $N$ of $Lv$. A natural question is, given a vertex $u$, can we
estimate the number\ $\kappa_{u}$ of distinct maximal left $L$-submodules $M$
of $Lu$ such that $Lu/M$ is isomorphic to $Lv/N$? In \cite{R-4} it is shown
that $\kappa_{u}\leq|uLv\backslash N|$ \ and consequently the cardinality of
the set of \ all such simple modules corresponding to various vertices is
$\leq%
%TCIMACRO{\dsum \limits_{u\in E^{0}}}%
%BeginExpansion
{\displaystyle\sum\limits_{u\in E^{0}}}
%EndExpansion
|uLv\backslash N|$ and thus is $\leq|L|$.

More generally, one may try to estimate the number of non-isomorphic simple
modules over a Leavitt path algebra. In this connection, observe the
following: Suppose a graph $E$ contains two cycles $g,h$ which share a common
vertex $v$, such as the two cycles in the following graph.%
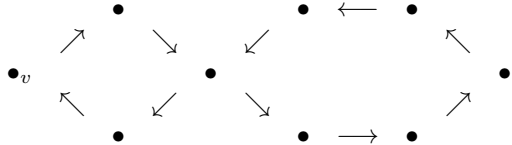
\[%
\begin{array}
[c]{ccccccccccc}
&  & \bullet &  &  &  & \bullet & \longleftarrow & \bullet &  & \\
& \nearrow &  & \searrow &  & \swarrow &  &  &  & \nwarrow & \\
\bullet_{v} &  &  &  & \bullet &  &  &  &  &  & \bullet\\
& \nwarrow &  & \swarrow &  & \searrow &  &  &  & \nearrow & \\
&  & \bullet &  &  &  & \bullet & \longrightarrow & \bullet &  &
\end{array}
\]
We then wish to show that there are uncountably many non-isomorphic simple
modules over the corresponding $L_{K}(E)$. By Theorem \ref{Chen}, we need only
to produce uncountably many non-equivalent infinite paths in $E$. With that in
mind, consider the infinite rational path $p=ggg\cdot\cdot\cdot$ which, for
convenience, we write as $p=g_{1}g_{2}g_{3}\cdot\cdot\cdot$ $\ $\ indexed by
the set $\mathbb{P}$ of positive integers, where $g_{i}=g$ for all $i$. Now,
for every subset $S$ of $\mathbb{P}$, define an infinite path $p_{S}$ by
replacing $g_{i}$ by $h$ if and only if $i\in S$. Observe that this gives rise
to uncountably many distinct infinite paths. From the definition of
equivalence of paths, it can be derived that, given an infinite path $q$ there
are at most countably many infinite paths that are equivalent to $q$. From
this one can then establish that there are uncountably many non-isomorphic
simple left $L_{K}(E)$-modules.

Leavitt path algebras having only finitely many non-isomorphic simple modules
are characterized in the next theorem. Recall, a ring $R$ is called left/right
semi-artinian if every non-zero left/right $R$-module contains a simple submodule.

\begin{theorem}
(\cite{A-2R})\label{FIRT} Let $E$ be an arbitrary graph and $K$ be any field.
Then the following are equivalent for the Leavitt path algebra $L=L_{K}(E)$:

(i) \ $L$ has at most finitely many non-isomorphic simple left/right $L$-modules;

(ii) $L$ is a left and right semi-artinian von Neumann regular ring with
finitely many two-sided ideals which form a chain under set inclusion;

(iii) The graph $E$ is acyclic and there is a finite ascending chain of
hereditary saturated subsets $\{0\}\subsetneqq H_{1}\subsetneqq\cdot\cdot
\cdot\subsetneqq H_{n}=E^{0}$ such that for each $i<n$, $H_{i+1}\backslash
H_{i}$ is the hereditary satusrated closure of the set of all the line points
in $E\backslash H_{i}$.
\end{theorem}

\section{One-sided Ideals in a Leavitt Path Algebra}

In this section, we shall describe some of the interesting properties of
one-sided ideals of a Leavitt path algebra.

Using a deep theorem of G. Bergman \cite{B}, Ara and Goodearl proved the
following result that Leavitt path algebras are hereditary.

\begin{theorem}
\label{hereditary}(\cite{AG}) If $\ E$ is an arbitrary graph, then every
left/right ideal of $L_{K}(E)$ is a projective left/right $L_{K}(E)$-module.
\end{theorem}

A von Neumann regular ring $R$ has the characterizing property that every
finitely generated one-sided ideal of $R$ is a principal ideal generated by an
idempotent. Four years ago, it was shown in \cite{R-2} that every finitely
generated two-sided ideal of a Leavitt path algebra is a principal ideal.
Recently, Abrams, Mantese and Tonolo have proved that this interesting
property holds for one sided ideals too, as indicated in the next theorem.

\begin{theorem}
\label{Bezout} (\cite{AMT-2}) Let $E$ be an arbitrary graph. Then
$L:=L_{K}(E)$ is a B\'{e}zout ring, that is, every finitely generated
one-sided ideal is a principal ideal.
\end{theorem}

Thus if $La$ and $Lb$ are two principal left ideals of $L$, then, being
finitely generated, $La+Lb=Lc$, a principal left ideal. So the sum of any two
principal left ideals of $L$ is again a principal left ideal. What about their
intersection? Should the intersection of two principal left ideals of $L$ be
again a principal left ideal? This is answered in the next theorem. Since this
result is new, we outline its proof which is straight forward.

\begin{theorem}
Let $E$ be an arbitrary graph. Then both the sum and the intersection of two
principal left ideal of $L:=L_{K}(E)$ are again principal left ideals. Thus
the principal left ideals of $L$ form a sublattice of the lattice of all left
ideals of $L$.
\end{theorem}

\begin{proof}
Suppose $A,B$ are two principal left ideals of $L$. Consider the following
exact sequence where the map $\theta$ is the additive map $(a,b)\rightarrow
a+b$%
\[
0\longrightarrow K\longrightarrow A\oplus B\overset{\theta}{\longrightarrow
}A+B\longrightarrow0
\]
where $K=\{(x,-x):x\in A\cap B\}\cong A\cap B$. Now $A+B$ is a (finitely
generated) left ideal of $L_{K}(E)$. By Theorem \ref{hereditary}, it is a
projective module and hence the above exact sequence splits. Consequestly,
$A\cap B$ is isomorphic a direct summand of $A\oplus B$. Since $A\oplus B$ is
finitely generated, so is $A\cap B$. By Theorem \ref{Bezout}, $A\cap B$ is
then a principal left ideal.
\end{proof}

As we noted in Theorem \ref{Condn K}, every two-sided ideal of $L_{K}(E)$ is
graded if and only if the graph $E$ satisfies Condition (K). What happens if
every one-sided ideal of $L_{K}(E)$ is graded ? This is answered in the next
and the last theorem of this section.

\begin{theorem}
(\cite{HR}) Let $E$ be an arbitrary graph. Then the following properties are
equivalent for $L=L_{K}(E)$:

(i) \ \ Every left ideal of $L$ is a graded left ideal;

(ii) \ Every simple left $L$-module is a graded module;

(iii) The graph $E$ contains no cycles;

(iv) $L$ is a von Neumann regular ring.
\end{theorem}

\end{document}